\documentclass[12pt,twoside]{amsart}
\usepackage[latin1]{inputenc}
\usepackage{amsmath, amsthm, amscd, amsfonts, amssymb, graphicx}
\usepackage[bookmarksnumbered, plainpages]{hyperref}

\newtheorem{thm}{Theorem}[section]

\newtheorem{lem}[thm]{Lemma}
\newtheorem{prop}[thm]{Proposition}
\newtheorem{defn}[thm]{Definition}

\newtheorem{ex}[thm]{Example}
\numberwithin{equation}{section}

\allowdisplaybreaks

\begin{document}

\title{\bf Riemannian geometry of noncommutative super surfaces}

\author{Yong Wang$^*$, Tong Wu}

\thanks{{\scriptsize
\hskip -0.4 true cm \textit{2010 Mathematics Subject Classification:}
53C40; 53C42.
\newline \textit{Corresponding author:} Yong Wang}}

\maketitle

\begin{abstract}
 In this paper, a Riemannian geometry of noncommutative super surfaces is developed which generalizes \cite{CTZZ} to the super case.
  The notions of metric and connections on such noncommutative super surfaces are introduced and it
is shown that the connections are metric-compatible and have zero torsion when the super metric is symmetric, giving rise to the corresponding super Riemann curvature. The latter also satisfies the noncommutative super analogue of the first and second Bianchi identities. We also give some examples and study them in details.

\end{abstract}

\vskip 0.2 true cm


\pagestyle{myheadings}
\markboth{\rightline {\scriptsize Wang}}
         {\leftline{\scriptsize Riemannian geometry of noncommutative super surfaces}}

\bigskip
\bigskip


\section{ Introduction}
 \indent  \quad
 \label{Sec:1}
  It is well known that $2$-dimensional surfaces embedded in the Euclidean $3$-space provide the simplest
yet nontrivial examples of Riemannian geometry. The Euclidean metric of the $3$-space induces a
natural metric for a surface through the embedding; the Levi-Civita connection and the curvature
of the tangent bundle of the surface can thus be described explicitly. In \cite{CTZZ}, Chaichian-Tureanu-Zhang-Zhang developed noncommutative deformations of Riemannian geometry in the light of
Whitney's theorem. They deformed the algebra of functions on a domain of the Euclidean space by introducing the Moyal algebra, which is a noncommutative deformation of the algebra of smooth functions on a region of $\mathbb{R}^2$. Then they developed a noncommutative Riemannian geometry for noncommutative analogues of $2$-dimensional surfaces embedded in $3$-space. Working over the Moyal algebra, they showed that
much of the classical differential geometry for surfaces generalized naturally to this noncommutative setting. In \cite{AB}, the authors constructed deformation of the algebra of diffeomorphisms for canonically deformed spaces with constant deformation parameter theta. The algebraic relations remained the same, whereas the comultiplication rule (Leibniz rule) was different from the undeformed one. Based on this deformed algebra a covariant tensor calculus was constructed and all the concepts like metric, covariant derivatives, curvature and torsion was defined on the deformed space as well. The construction of these geometric quantities was presented in detail. In \cite{AB1}, Aschier found that the Lie algebra of infinitesimal diffeomorphisms on noncommutative space allowed to develop differential and Riemannian noncommutative geometry. And noncommutative Einstein's gravity equations were formulated. In \cite{Go}, Goertsches developed a theory of Riemannian supermanifolds up to a definition of Riemannian symmetric superspaces. And various fundamental concepts needed for the study of these spaces both from the Riemannian and the Lie theoretical viewpoint were introduced.\\
 \indent On the other hand, it is well known that the classical differential geometry can be generalized to the super case. In \cite{BG}, Bruce and Grabowski examined the notion of a Riemannian $\mathbb{Z}_2^n$ manifold.
 They showed that the basic notions and tenets of Riemannian geometry directly generalized to the setting of $\mathbb{Z}_2^n$-geometry.
For example, the fundamental theorem holded in the higher graded setting. They pointed out the
similarities and differences with Riemannian supergeometry.\\
 \indent The motivation of this paper is to generalize \cite{CTZZ} to the super case. In Section \ref{Sec:2}, we introduce the super Moyal algebra
and the notions of metric and connections on noncommutative super surfaces and it
is shown that the connections are metric-compatible and have zero torsion when the super metric is symmetric, giving rise to the corresponding super Riemann curvature. The latter also satisfies the noncommutative super analogue of the first and second Bianchi identities.
 In Section \ref{Sec:3}, we give some examples and study them in details.


\vskip 1 true cm

\section{Noncommutative super surfaces}
\label{Sec:2}
Firstly we introduce some notations on Riemannian supergeometry.
\begin{defn}\label{def1} A locally $\mathbb{Z}_2$-ringed space is a pair $S:= (|S|, \mathcal{O}_S)$ where $|S|$ is a second-countable
Hausdorff space, and a $\mathcal{O}_S$ is a sheaf of $\mathbb{Z}_2$-graded $\mathbb{Z}_2$-commutative associative unital $\mathbb{R}$-algebras, such that the
stalks $\mathcal{O}_{S,p}$, $p\in |S|$ are local rings.
\end{defn}
 \indent In this context, $\mathbb{Z}_2$-commutative means that any two sections $s,t\in \mathcal{O}_S(|U|),~~|U|\subset|S|$ open,
 of homogeneous
degree $|s|\in \mathbb{Z}_2$
and $|t|\in \mathbb{Z}_2$
commute up to the sign rule
$st=(-1)^{|s||t|}ts$.
 $\mathbb{Z}_2$-ring
space $U^{m|n}:= (U,C^{\infty}_{U^m}\otimes \wedge \mathbb{R}^n)$, is called standard
superdomain where $C^{\infty}_{U^m}$ is the sheaf of smooth functions on $U$ and $\wedge\mathbb{R}^n$ is
the exterior algebra of $\mathbb{R}^n$. We can employ (natural) coordinates $x^I:=(x^a,\xi^A)$ on any $\mathbb{Z}_2$-domain, where $x^a$ form a coordinate system on $U$ and the $\xi^A$
are formal coordinates.
\begin{defn}\label{def2}
 A supermanifold of dimension $m|n$ is a super ringed space
$M=(|M|, \mathcal{O}_M )$ that is locally isomorphic to $\mathbb{R}^{m|n}$ and $|M|$ is a second countable
and Hausdorff topological space.
\end{defn}
 The tangent sheaf $\mathcal{T}M$ of a $\mathbb{Z}_2$-manifold $M$ is defined as the sheaf of derivations of sections of the structure
sheaf, i.e., $\mathcal{T}M(|U|) := {\rm Der}(\mathcal{O}_M(|U|)),$ for arbitrary open set $|U|\subset |M|.$ Naturally, this is a sheaf of locally free $\mathcal{O}_M$-modules. Global sections of the tangent sheaf are referred to as {\it vector fields}. We denote the $\mathcal{O}_M (|M|)$-module
of vector fields as ${\rm Vect}(M)$. The dual of the tangent sheaf is the {\it cotangent sheaf}, which we denote as $\mathcal{T}^*M$.
This is also a sheaf of locally free $\mathcal{O}_M$-modules. Global section of the cotangent sheaf we will refer to as {\it one-forms}
and we denote the $\mathcal{O}_M(|M|)$-module of one-forms as $\Omega^1(M)$.
\begin{defn}\label{def3}
 A Riemannian metric on a $\mathbb{Z}_2$-manifold M is a $\mathbb{Z}_2$-homogeneous, $\mathbb{Z}_2$-symmetric, non-degenerate,
$\mathcal{O}_M$-linear morphisms of sheaves $\left<-,-\right>_g:~~\mathcal{T}M\otimes \mathcal{T}M\rightarrow \mathcal{O}_M.$
A $\mathbb{Z}_2$-manifold equipped with a Riemannian metric is referred to as a Riemannian $\mathbb{Z}_2$-manifold.
\end{defn}
We will insist that the Riemannian metric is homogeneous with respect to the $\mathbb{Z}_2$-degree, and we will denote
the degree of the metric as $|g| \in \mathbb{Z}_2$.
Explicitly, a Riemannian metric has the following properties:\\
(1)$ |\left<X,Y\right>_g |= |X| + |Y |+ |g|,$\\
(2)$\left<X,Y\right>_g =(-1)^{|X||Y|}\left<Y,X\right>_g,$\\
(3) If $\left<X,Y\right>_g = 0$ for all $Y \in Vect(M),$ then $X = 0,$\\
(4) $\left<fX+Y,Z\right>_g =f\left<X,Z\right>_g +\left<Y,Z\right>_g ,$\\
for arbitrary (homogeneous) $ X, Y, Z \in {\rm Vect}(M)$ and $f \in C^{\infty}(M)$. We will say that a Riemannian metric is
even if and only if it has degree zero. Similarly, we will say that a Riemannian metric is odd if and only
if it has degree one. Any Riemannian metric we consider will be either even or odd as we will only be
considering homogeneous metrics.\\
\indent Similar to \cite{CTZZ}, we give some notions about noncommutative super surfaces. Let us fix a region $U$ in $R^2$ and write the coordinate of a point $t$ in $U$ as $(t_1 ,t_2)$. Let $h$ be a real indeterminate, and denote by $R[[\overline{h}]]$ the ring of formal power
series in $\overline{h}$. Let $\mathcal{A}$ be the set of the formal power series in $\overline{h}$ with coefficients being real smooth functions on $U$. Namely, every element of $\mathcal{A}$ is of the form $\sum_{i\geq0}f_i\overline{h^i}$, where $f_i$ are smooth functions on $U$. Then $\mathcal{A}$ is an $R[[\overline{h}]]$-module in an obvious way. Let $\bigwedge^P(\xi_1,\cdot\cdot\cdot,\xi_r)$ be Grassmann algebra, we take $a,b \in \mathcal{A}\otimes\bigwedge$, then $a=\sum_{1\leq i_1\leq\cdot\cdot\cdot\leq i_k\leq P}f_{i_1\cdot\cdot\cdot i_k}\xi^{i_1}\cdot\cdot\cdot\xi^{i_k},$ $b=\sum_{1\leq j_1\leq\cdot\cdot\cdot\leq j_q\leq P}g_{j_1\cdot\cdot\cdot j_q}\xi^{j_1}\cdot\cdot\cdot\xi^{j_q},$ where $f_{i_1\cdot\cdot\cdot i_k}, g_{j_1\cdot\cdot\cdot j_q} \in \mathcal{A}.$ Define their star product (or more precisely, Moyal product)
\begin{align}\label{a1}
a\ast b:=\sum_{1\leq i_1\leq\cdot\cdot\cdot\leq i_k\leq P,1\leq j_1\leq\cdot\cdot\cdot\leq j_q\leq P}f_{i_1\cdot\cdot\cdot i_k}\ast g_{j_1\cdot\cdot\cdot j_q}\xi^{i_1}\wedge\cdot\cdot\cdot\wedge\xi^{i_k}\wedge\xi^{j_1}\cdot\cdot\cdot\wedge\xi^{j_q},\nonumber\\
\end{align}
where $f_{i_1\cdot\cdot\cdot i_k}\ast g_{j_1\cdot\cdot\cdot j_q}$ is the star product of $f_{i_1\cdot\cdot\cdot i_k}$ and $g_{j_1\cdot\cdot\cdot j_q}$ in $\mathcal{A}$.

Obviously, star product in $\mathcal{A}\otimes\bigwedge$ is associative. For the following part, we will denote $\mathcal{A}\otimes\bigwedge$ by $\widetilde{\mathcal{A}}.$

Let $x^I:=(t_1 ,t_2,\xi_1,\cdot\cdot\cdot,\xi_P),$ then
\begin{align}\label{a2}
\partial_{x^I}(a\ast b)=(\partial_{x^I}a)\ast b+(-1)^{|\partial_{x^I}||a|}a\ast\partial_{x^I}b,\nonumber\\
\end{align}
where the operators $\partial_{x^I}$ are derivations of the algebra $\widetilde{\mathcal{A}}$, $a,b$ are homogeneous and $|\partial_{t_1}|=|\partial_{t_2}|=0,~|\partial_{\xi_\alpha}|=1~(1\leq\alpha \leq P).$

\begin{defn}\label{def5}
Let $TX=\widetilde{\mathcal{A}}\{\frac{\partial}{\partial t_1},\frac{\partial}{\partial t_2},\frac{\partial}{\partial \xi^1},\cdot\cdot\cdot,\frac{\partial}{\partial \xi^P}\}$ be the free-left $\widetilde{\mathcal{A}}$-module and $\widetilde{T}X=\{\frac{\partial}{\partial t_1},\frac{\partial}{\partial t_2},\frac{\partial}{\partial \xi^1},\cdot\cdot\cdot,\frac{\partial}{\partial \xi^P}\}\widetilde{\mathcal{A}}$ be the free-right $\widetilde{\mathcal{A}}$-module. Let $g:TX\otimes \widetilde{T}X\rightarrow \widetilde{\mathcal{A}}$ be a double $\widetilde{\mathcal{A}}$ module map defined by $g_{IJ}=\left<\partial_{x^I},\partial_{x^J}\right>_g,$ and let $g^0=g~mod~\overline{h}$, which is a inverse matrix of smooth functions on $U\otimes\bigwedge^P$, then we call $g$ the metric of the
noncommutative super surface $TX$.
\end{defn}

Given a noncommutative super surface $TX$ with a metric g, there exists a unique matrix $[g^{JK}]$
over $\widetilde{\mathcal{A}}$, which is the right inverse of $g$, i.e.,
$$g_{IJ}\ast g^{JK}=\delta_I^K,$$
where we have used Einstein's convention of summing over repeated indices.

\begin{defn}\label{def4} For $a,b,c_i~(1\leq i\leq P),$ then $\partial_{x^I}$ in $\widetilde{\mathcal{A}}$ generate the left $\widetilde{\mathcal{A}}$-module $TX$ and right $\widetilde{\mathcal{A}}$-module $\widetilde{T}X,$ defined by
 $$TX=a\ast\partial_{t_1}+b\ast\partial_{t_2}+\sum_{i=1}^Pc_i\ast\partial_{\xi_i},~~~\widetilde{T}X=\partial_{t_1}\ast a+\partial_{t_2}\ast b+\sum_{i=1}^P\partial_{\xi_i}\ast c_i.$$
That is the left and right tangent bundles of the noncommutative super surface respectively, we call also $TX$ a noncommutative super surface.
\end{defn}

\begin{prop}\label{prop1}
The metric induces a homomorphism of two-sided $\widetilde{\mathcal{A}}$-modules,
$$g: TX\otimes_{R[[\overline{h}]]} \widetilde{T}X\longrightarrow\widetilde{ \mathcal{A}}$$
defined for any $Z=Z^I\ast\partial_{x^I}\in TX$ and $\widetilde{Z}=\partial_{x^I}\ast \widetilde{Z}^I\in \widetilde{T}X$ by
$$Z\otimes \widetilde{Z}\mapsto \left<Z,\widetilde{Z}\right>_g=Z^I\ast g_{IJ}\ast \widetilde{Z}^J.$$
\end{prop}
We have $\left<a\partial_{x^I},\partial_{x^J}b\right>_g=a\ast g_{IJ}\ast b$ and $|\left<\partial_{x^I},\partial_{x^J}\right>_g|=|g|+|\partial_{x^I}|+|\partial_{x^J}|.$ Here we don't assume that $g_{IJ}=(-1)^{|\partial_{x^I}||\partial_{x^J}|}g_{JI}.$

Next we define the Levi-Civita connections $\nabla$ and $\widetilde{\nabla}$. We define $\Gamma_{IJ}^L$ and $\widetilde{\Gamma}_{IJ}^L$ in $\widetilde{\mathcal{A}}$ such that
\begin{align}\label{a7000}
&\Gamma_{IJ}^L:=\Gamma_{IJK}\ast g^{KL},~~~\Gamma_{IJK}:=\frac{1}{2}\left(\frac{\partial {g_{JK}}}{\partial{x^I}}+(-1)^{|\partial_{x^I}||\partial_{x^J}|}\frac{\partial {g_{IK}}}{\partial{x^J}}-(-1)^{|\partial_{x^K}|(|\partial_{x^I}|+|\partial_{x^J}|)}\frac{\partial {g_{IJ}}}{\partial{x^K}}\right);\nonumber\\
&\widetilde{\Gamma}_{IJ}^L:=g^{KL}\ast \widetilde{\Gamma}_{IJL},~~~\widetilde{\Gamma}_{IJL}=(-1)^{|\partial_{x^L}|(|\partial_{x^I}|+|\partial_{x^J}|)}\Gamma_{IJL};\nonumber\\
&\nabla_{\partial_{x^I}}\partial_{x^J}=\Gamma_{IJ}^L\partial_{x^L},~~~\widetilde{\nabla}_{\partial_{x^I}}\partial_{x^J}=\partial_{x^L}\widetilde{\Gamma}_{IJ}^L.\nonumber\\
\end{align}
Define
$$\nabla_{\partial_{x^I}}(f\partial{x^J})=(\partial_{x^I}f)\partial{x^J}+(-1)^{|f||\partial{x^I}|}f\ast \nabla_{\partial_{x^I}}\partial_{x^J}$$
and
$$\widetilde{\nabla}_{\partial_{x^I}}(\partial{x^J}f)=(-1)^{|\partial{x^J}||\partial{x^I}|}\partial_{x^J}(\partial_{x^I}f)+ \widetilde{\nabla}_{\partial_{x^I}}\partial{x^J}\ast f.$$
Then we get the following lemma
\begin{lem}\label{lem1}
For all $Z\in TX,$ $\widetilde{Z}\in \widetilde{T}X$ and $f\in \widetilde{\mathcal{A}},$
\begin{align}\label{a4}
&\nabla_{\partial_{x^I}}(f\ast Z)=(\partial_{x^I}f)\ast Z+(-1)^{|f||\partial_{x^I}|}f\ast\nabla_{\partial_{x^I}} Z,\nonumber\\
&\widetilde{\nabla}_{\partial_{x^I}}(\widetilde{Z}\ast f)=(\widetilde{\nabla}_{\partial_{x^I}}\widetilde{Z})\ast f+(-1)^{|\widetilde{Z}||\partial_{x^I}|}\widetilde{Z}\ast \partial_{x^I}f.\nonumber\\
\end{align}
\end{lem}
\begin{proof}
Let $Z=f_1\partial_{x^J}$, $\widetilde{Z}=\partial_{x^J}f_1$, we have
\begin{align}\label{a5}
\nabla_{\partial_{x^I}}(f\ast Z)&=\nabla_{\partial_{x^I}}[(f\ast f_1)\partial_{x^J}]\nonumber\\
&=\partial_{x^I}(f\ast f_1)\partial_{x^J}+(-1)^{(|f|+|f_1|)|\partial_{x^I}|}(f\ast f_1)\ast \nabla_{\partial_{x^I}}{\partial_{x^J}}\nonumber\\
&=[(\partial_{x^I}f)\ast f_1+(-1)^{|f||\partial_{x^I}|}f\ast \partial_{x^I}{f_1}]\partial_{x^J}+(-1)^{(|f|+|f_1|)|\partial_{x^I}|}(f\ast f_1)\ast \nabla_{\partial_{x^I}}{\partial_{x^J}}\nonumber\\
&=(\partial_{x^I}f)\ast (f_1\partial_{x^J})+(-1)^{|f||\partial_{x^I}|}[(f\ast \partial_{x^I}{f_1})\partial_{x^J}+(-1)^{|f_1||\partial_{x^I}|}f\ast f_1\ast \nabla_{\partial_{x^I}}{\partial_{x^J}}]\nonumber\\
&=(\partial_{x^I}f)\ast Z+(-1)^{|f||\partial_{x^I}|}f\ast\nabla_{\partial_{x^I}} Z,\nonumber\\
\end{align}
and
\begin{align}\label{a6}
\widetilde{\nabla}_{\partial_{x^I}}(\widetilde{Z}\ast f)&=\widetilde{\nabla}_{\partial_{x^I}}(\partial_{x^J}f_1\ast f)\nonumber\\
&=(-1)^{|\partial_{x^I}||\partial_{x^J}|}\partial_{x^J}\ast[\partial_{x^I}(f_1\ast f)]+\widetilde{\nabla}_{\partial_{x^I}}{\partial_{x^J}}\ast (f_1\ast f)\nonumber\\
&=(-1)^{|\partial_{x^I}||\partial_{x^J}|}\partial_{x^J}\ast[\partial_{x^I}f_1\ast f+(-1)^{|\partial_{x^I}||f_1|}f_1\ast \partial_{x^I}f]+(\widetilde{\nabla}_{\partial_{x^I}}{\partial_{x^J}}\ast f_1)\ast f\nonumber\\
&=(-1)^{|\partial_{x^I}||\partial_{x^J}|}\partial_{x^J}\ast(\partial_{x^I}f_1)\ast f+(-1)^{|\partial_{x^I}|(|\partial_{x^J}|+|f_1|)}(\partial_{x^J}\ast f_1)\ast \partial_{x^I}f\nonumber\\
&+(\widetilde{\nabla}_{\partial_{x^I}}{\partial_{x^J}}\ast f_1)\ast f\nonumber\\
&=(\widetilde{\nabla}_{\partial_{x^I}}\widetilde{Z})\ast f+(-1)^{|\widetilde{Z}||\partial_{x^I}|}\widetilde{Z}\ast \partial_{x^I}f,\nonumber\\
\end{align}
then we get (\ref{a4}).
\end{proof}

\begin{prop}\label{prop2}
For $Z\in TX,$ $\widetilde{Z}\in \widetilde{T}X$, when $|g|=0,$ $g_{JK}=(-1)^{|\partial_{x^J}||\partial_{x^K}|}g_{KJ},$ then the connections are metric compatible in the following sense:
\begin{align}\label{a8}
\partial_{x^I}\left<Z,\widetilde{Z}\right>_g=\left<\nabla_{\partial_{x^I}}Z,\widetilde{Z}\right>_g+(-1)^{|\partial_{x^I}||Z|}\left<Z,\widetilde{\nabla}_{\partial_{x^I}}\widetilde{Z}\right>_g.\nonumber\\
\end{align}
\end{prop}
\begin{proof}
Let $Z=f_1\partial_{x^J},~\widetilde{Z}=\partial_{x^K}f_2,$ then
\begin{align}\label{a7003}
\partial_{x^I}\left<Z,\widetilde{Z}\right>_g&=\partial_{x^I}\left<f_1\partial_{x^J},\partial_{x^K}f_2\right>_g\nonumber\\
&=\partial_{x^I}(f_1\ast g_{JK}\ast f_2)\nonumber\\
&=\partial_{x^I}f_1\ast g_{JK}\ast f_2+(-1)^{|\partial_{x^I}||f_1|}f_1\ast \partial_{x^I}g_{JK}\ast f_2\nonumber\\
&+(-1)^{|\partial_{x^I}|(|f_1|+|g_{JK}|)}f_1\ast g_{JK}\ast \partial_{x^I}f_2,\nonumber\\
\end{align}
and
\begin{align}\label{a7004}
&\left<\nabla_{\partial_{x^I}}Z,\widetilde{Z}\right>_g+(-1)^{|\partial_{x^I}||Z|}\left<Z,\widetilde{\nabla}_{\partial_{x^I}}\widetilde{Z}\right>_g\nonumber\\
&=\left<\nabla_{\partial_{x^I}}(f_1\partial_J),\partial_{x^K}f_2\right>_g+(-1)^{|\partial_{x^I}||(f_1\partial_{x^J})|}\left<(f_1\partial_{x^J}),\widetilde{\nabla}_{\partial_{x^I}}(\partial_{x^K}f_2)\right>_g\nonumber\\
&=(\partial_{x^I}f_1)\ast g_{JK}\ast f_2+(-1)^{|\partial_{x^I}||f_1|}f_1\ast\left<\nabla_{\partial_{x^I}}\partial_{x^J},\partial_{x^K}\right>_g\ast f_2+(-1)^{|\partial_{x^I}|(|f_1|+|\partial_{x^J}|)}\nonumber\\
&(-1)^{|\partial_{x^K}||\partial_{x^I}|}f_1\ast g_{JK}\ast \partial_{x^I}f_2+(-1)^{|\partial_{x^I}|(|f_1|+|\partial_{x^J}|)}f_1\left<\partial_{x^J},\widetilde{\nabla}_{\partial_{x^I}}\partial_{x^K}\right>_g\ast f_2.\nonumber\\
\end{align}
Because by $|g|=0$,
\begin{align}\label{a7005}
(-1)^{|\partial_{x^I}|(|f_1|+|\partial_{x^J}|)}(-1)^{|\partial_{x^K}||\partial_{x^I}|}=(-1)^{|\partial_{x^I}|(|f_1|+|g_{JK})|}.\nonumber\\
\end{align}
We only prove
\begin{align}\label{a7006}
\left<\nabla_{\partial_{x^I}}\partial_{x^J},\partial_{x^K}\right>_g+(-1)^{|\partial_{x^I}||\partial_{x^J}|}\left<\partial_{x^J},\widetilde{\nabla}_{\partial_{x^I}}\partial_{x^K}\right>_g= \partial_{x^I}g_{JK}.\nonumber\\
\end{align}
Then by (\ref{a7000}), we have
\begin{align}\label{a7007}
&\left<\nabla_{\partial_{x^I}}\partial_{x^J},\partial_{x^K}\right>_g+(-1)^{|\partial_{x^I}||\partial_{x^J}|}\left<\partial_{x^J},\widetilde{\nabla}_{\partial_{x^I}}\partial_{x^K}\right>_g\nonumber\\
&=\left<\Gamma_{IJ}^L\partial_{x^L},\partial_{x^K}\right>_g+(-1)^{|\partial_{x^I}||\partial_{x^J}|}\left<\partial_{x^J},\partial_{x^L}\widetilde{\Gamma}_{IK}^L\right>_g\nonumber\\
&=\Gamma_{IJ}^L\ast g_{JK}+(-1)^{|\partial_{x^I}||\partial_{x^J}|}g_{JL}\ast \widetilde{\Gamma}_{IK}^L\nonumber\\
&=\Gamma_{IJK}+(-1)^{|\partial_{x^I}||\partial_{x^J}|}\widetilde{\Gamma}_{IKJ}.\nonumber\\
\end{align}
By (\ref{a7000})-(\ref{a7003}) and $\widetilde{\Gamma}_{IKJ}=(-1)^{|\partial_{x^J}|(|\partial_{x^I}|+|\partial_{x^K}|)}\Gamma_{IKJ},$ we have
\begin{align}\label{a7008}
&\Gamma_{IJK}+(-1)^{|\partial_{x^I}||\partial_{x^J}|}\widetilde{\Gamma}_{IKJ}\nonumber\\
&=\frac{1}{2}\left(\frac{\partial {g_{JK}}}{\partial{x^I}}+(-1)^{|\partial_{x^I}||\partial_{x^J}|}\frac{\partial {g_{IK}}}{\partial{x^J}}-(-1)^{|\partial_{x^K}|(|\partial_{x^I}|+|\partial_{x^J}|)}\frac{\partial {g_{IJ}}}{\partial{x^K}}\right)+(-1)^{|\partial_{x^J}||\partial_{x^K}|}\nonumber\\
&+\frac{1}{2}\left(\frac{\partial{g_{KJ}}}{\partial{x^I}}+(-1)^{|\partial_{x^I}||\partial_{x^K}|}\frac{\partial {g_{IJ}}}{\partial{x^J}}-(-1)^{|\partial_{x^J}|(|\partial_{x^I}|+|\partial_{x^K}|)}\frac{\partial {g_{IK}}}{\partial{x^K}}\right)\nonumber\\
&=\frac{1}{2}\partial_{x^I}[g_{JK}+(-1)^{|\partial_{x^J}||\partial_{x^K}|}g_{KJ}]\nonumber\\
&=\partial_{x^I}g_{JK}.\nonumber\\
\end{align}
Therefore, (\ref{a8}) holds.
\end{proof}
\begin{prop}\label{prop3}
When $g_{IJ}=(-1)^{|\partial_{x^J}||\partial_{x^I}|}g_{JI},$ then torsion vanishes in the following sense:
\begin{align}\label{a9}
T^\nabla=0,~~~T^{\widetilde{\nabla}}=0.
\end{align}
\end{prop}
\begin{proof}
By \cite{BG}, we have
\begin{align}\label{a10}
T^\nabla(\partial_{x^I},\partial_{x^J})&=\nabla_{\partial_{x^I}}{\partial_{x^J}}-(-1)^{|\partial_{x^J}||\partial_{x^I}|}\nabla_{\partial_{x^J}}{\partial_{x^I}}\nonumber\\
&=\Gamma_{IJ}^L\partial_{x^L}-(-1)^{|\partial_{x^J}||\partial_{x^I}|}\Gamma_{JI}^L\partial_{x^L}\nonumber\\
&=(\Gamma_{IJ}^L-(-1)^{|\partial_{x^J}||\partial_{x^I}|}\Gamma_{JI}^L)\partial_{x^L}.\nonumber\\
\end{align}
By (\ref{a7000}), we have
\begin{align}\label{a11}
&\Gamma_{IJ}^L-(-1)^{|\partial_{x^J}||\partial_{x^I}|}\Gamma_{JI}^L\nonumber\\
&=\frac{1}{2}\left(\frac{\partial {g_{JK}}}{\partial_{x^I}}+(-1)^{|\partial_{x^I}||\partial_{x^J}|}\frac{\partial {g_{IK}}}{\partial_{x^J}}-(-1)^{|\partial_{x^K}|(|\partial_{x^I}|+|\partial_{x^J|)}}\frac{\partial {g_{IJ}}}{\partial_{x^K}}\right)\ast g^{KL}\nonumber\\
&-(-1)^{|\partial_{x^J}||\partial_{x^I}|}\frac{1}{2}\left(\frac{\partial {g_{IK}}}{\partial_{x^J}}+(-1)^{|\partial_{x^I}||\partial_{x^J}|}\frac{\partial {g_{JK}}}{\partial_{x^I}}-(-1)^{|\partial_{x^K}|(|\partial_{x^I}|+|\partial_{x^J}|)}\frac{\partial {g_{JI}}}{\partial_{x^K}}\right)\ast g^{KL}\nonumber\\
&=(-1)^{|\partial_{x^K}|(|\partial_{x^I}|+|\partial_{x^J}|)}\left((-1)^{|\partial_{x^I}||\partial_{x^J}|}\frac{\partial {g_{JI}}}{\partial_{x^K}}-\frac{\partial {g_{IJ}}}{\partial_{x^K}}\right)\nonumber\\
&=0.\nonumber\\
\end{align}
Similarly, $T^{\widetilde{\nabla}}(\partial_{x^I},\partial_{x^J})=0.$ Therefore, we get Proposition \ref{prop3}.
\end{proof}
Next we consider curvatures and Bianchi identities in noncommutative super surfaces.

Let $[\nabla_{\partial_{x^I}},\nabla_{\partial_{x^J}}]:=\nabla_{\partial_{x^I}}\nabla_{\partial_{x^J}}-(-1)^{|\partial_{x^I}||\partial_{x^J}|}\nabla_{\partial_{x^J}}\nabla_{\partial_{x^I}}$ and $[\widetilde{\nabla}_{\partial_{x^I}},\widetilde{\nabla}_{\partial_{x^J}}]:=\widetilde{\nabla}_{\partial_{x^I}}\widetilde{\nabla}_{\partial_{x^J}}-(-1)^{|\partial_{x^I}||\partial_{x^J}|}\widetilde{\nabla}_{\partial_{x^J}}\widetilde{\nabla}_{\partial_{x^I}}.$ Straightforward calculations show that
for all $f\in \widetilde{\mathcal{A}}$,
\begin{align*}
&[\nabla_{\partial_{x^I}},\nabla_{\partial_{x^J}}](f\ast Z)=(-1)^{|f|(|\partial_{x^I}|+|\partial_{x^J}|)}f\ast [\nabla_{\partial_{x^I}},\nabla_{\partial_{x^J}}]Z,~~~Z\in TX,\nonumber\\
&[\widetilde{\nabla}_{\partial_{x^I}},\widetilde{\nabla}_{\partial_{x^J}}](\widetilde{Z}\ast f)=[\widetilde{\nabla}_{\partial_{x^I}},\widetilde{\nabla}_{\partial_{x^J}}]\widetilde{Z}\ast f,~~~\widetilde{Z}\in \widetilde{T}X.\nonumber\\
\end{align*}
Clearly the right-hand side of the first equation belongs to $TX$, while that of the second equation
belongs to $\widetilde{T}X$. We restate these important facts as a proposition.
\begin{prop}\label{prop4}
The following maps
\begin{align*}
[\nabla_{\partial_{x^I}},\nabla_{\partial_{x^J}}]:TX\rightarrow TX,~~~[\widetilde{\nabla}_{\partial_{x^I}},\widetilde{\nabla}_{\partial_{x^J}}]:\widetilde{T}X\rightarrow \widetilde{T}X.
\end{align*}
are super left and super right $\widetilde{\mathcal{A}}$-module homomorphisms, respectively
\end{prop}
Write
\begin{align}\label{a12}
&[\nabla_{\partial_{x^I}},\nabla_{\partial_{x^J}}]\partial_{x^K}=R(\partial_{x^I},\partial_{x^J})\partial_{x^K}:=R^L_{IJK}\ast\partial_{x^L},\nonumber\\
&[\widetilde{\nabla}_{\partial_{x^I}},\widetilde{\nabla}_{\partial_{x^J}}]\partial_{x^K}=\widetilde{R}(\partial_{x^I},\partial_{x^J})\partial_{x^K}:=\partial_{x^L}\ast \widetilde{R}^L_{IJK},\nonumber\\
\end{align}
for some $R^L_{IJK},~\widetilde{R}^L_{IJK}\in \widetilde{\mathcal{A}}.$

\begin{defn}\label{def7}
We refer $R^L_{IJK},~\widetilde{R}^L_{IJK},$ respectively, as the super Riemann curvatures of the left
and right tangent bundles of the noncommutative super surface $X$.
\end{defn}
\begin{lem}\label{lem2}
The following equalities holds:
\begin{align*}
&R^L_{IJK}=\partial_{x^I}(\Gamma_{JK}^L)-(-1)^{|\partial_{x^I}||\partial_{x^J}|}\partial_{x^J}(\Gamma_{IK}^L)+(-1)^{|\partial_{x^I}|(|\partial_{x^J}|+|\partial_{x^K}|+|\partial_{x^M}|)}\Gamma^M_{JK}\ast \Gamma_{IM}^L\nonumber\\
&-(-1)^{|\partial_{x^J}|(|\partial_{x^K}|+|\partial_{x^M}|)}\Gamma^M_{IK}\ast\Gamma_{JM}^L,\\
&\widetilde{R}^L_{IJK}=\widetilde{\Gamma}^L_{IM}\ast\widetilde{\Gamma}_{JK}^M-(-1)^{|\partial_{x^I}||\partial_{x^J}|}\widetilde{\Gamma}^L_{JM}\ast\widetilde{\Gamma}_{IK}^M+(-1)^{|\partial_{x^I}||\partial_{x^L}|}\partial_{x^I}(\widetilde{\Gamma}_{JK}^L)\nonumber\\
&-(-1)^{|\partial_{x^J}|(|\partial_{x^I}|+|\partial_{x^L}|)}\partial_{x^J}(\widetilde{\Gamma}^L_{IK}).
\end{align*}
\end{lem}
\begin{prop}\label{prop5}
The super Riemann curvatures of the left
and right tangent bundles coincide in the sense that $R_{IJKL}=-(-1)^{(|\partial_{x^I}|+|\partial_{x^J}|)|\partial_{x^K}|}\widetilde{R}_{IJLK}.$
\end{prop}
\begin{proof}
Define $R_{IJKL}=R_{IJK}^M\ast g_{ML}$ and $\widetilde{R}_{IJKL}=g_{LM}\ast\widetilde{R}_{IJK}^M.$
Then,
\begin{align}\label{a14}
R_{IJKL}&=\left<(\nabla_{\partial_{x^I}}\nabla_{\partial_{x^J}}-(-1)^{|\partial_{x^I}||\partial_{x^J}|}\nabla_{\partial_{x^J}}\nabla_{\partial_{x^I}})\partial_{x^K},\partial_{x^I}\right>_g\nonumber\\
&=\left<\nabla_{\partial_{x^I}}\nabla_{\partial_{x^J}}\partial_{x^K},\partial_{x^L}\right>_g-(-1)^{|\partial_{x^I}||\partial_{x^J}|}\left<\nabla_{\partial_{x^J}}\nabla_{\partial_{x^I}}\partial_{x^K},\partial_{x^L}\right>_g\nonumber\\
&=\partial_{x^I}\left<\nabla_{\partial_{x^J}}\partial_{x^K},\partial_{x^L}\right>_g-(-1)^{|\partial_{x^I}|(|\partial_{x^J}|+|\partial_{x^K}|)}\left<\nabla_{\partial_{x^J}}\partial_{x^K},\widetilde{\nabla}_{\partial_{x^I}}\partial_{x^L}\right>_g-(-1)^{|\partial_{x^I}||\partial_{x^J}|}\nonumber\\
&\partial_{x^J}\left<\nabla_{\partial_{x^I}}\partial_{x^K},\partial_{x^L}\right>_g+(-1)^{|\partial_{x^I}||\partial_{x^J}|}(-1)^{|\partial_{x^J}|(|\partial_{x^I}|+|\partial_{x^K}|)}\left<\nabla_{\partial_{x^I}}\partial_{x^K},\widetilde{\nabla}_{\partial_{x^J}}\partial_{x^L}\right>_g.\nonumber\\
\end{align}
By Proposition \ref{prop2} and $$\left<\nabla_{\partial_{x^J}}\partial_{x^K},\widetilde{\nabla}_{\partial_{x^I}}\partial_{x^L}\right>_g=\partial_{x^J}\left<\partial_{x^K},\widetilde{\nabla}_{\partial_{x^I}}\partial_{x^L}\right>_g-(-1)^{|\partial_{x^K}||\partial_{x^J}|}\left<\partial_{x^K},\widetilde{\nabla}_{\partial_{x^J}}\widetilde{\nabla}_{\partial_{x^I}}\partial_{x^L}\right>_g,$$ then we have
\begin{align}\label{a15}
R_{IJKL}&=(-1)^{|\partial_{x^I}|(|\partial_{x^J}|+|\partial_{x^K}|)}(-1)^{|\partial_{x^J}||\partial_{x^K}|}\left<\partial_{x^K},\widetilde{\nabla}_{\partial_{x^J}}\widetilde{\nabla}_{\partial_{x^I}}\partial_{x^L}\right>_g\nonumber\\
&-(-1)^{|\partial_{x^K}|(|\partial_{x^I}|+|\partial_{x^J}|)}\left<\partial_{x^K},\widetilde{\nabla}_{\partial_{x^I}}\widetilde{\nabla}_{\partial_{x^J}}\partial_{x^L}\right>_g\nonumber\\
&=-(-1)^{(|\partial_{x^I}|+|\partial_{x^J}|)|\partial_{x^K}|}\widetilde{R}_{IJLK}.\nonumber\\
\end{align}
\end{proof}
Let \begin{align*}
R_{IJK;P}^L&=\partial_{x^P}(R_{IJK}^L)+(-1)^{|\partial_{x^P}|(|\partial_{x^I}|+|\partial_{x^J}|+|\partial_{x^K}|+|\partial_{x^S}|)}R_{IJK}^S\ast\Gamma_{PS}^L-\Gamma_{PI}^S\ast R_{SJK}^L\nonumber\\
&-(-1)^{|\partial_{x^I}|(|\partial_{x^S}|+|\partial_{x^J}|)}\Gamma_{PJ}^S\ast R_{ISK}^L-(-1)^{(|\partial_{x^I}|+|\partial_{x^J}|)(|\partial_{x^K}|+|\partial_{x^S}|)}\Gamma_{PK}^S\ast R_{IJS}^L.
\end{align*}
Then we get the following theorem
\begin{thm}\label{thm1}
When $g_{IJ}=(-1)^{|\partial_{x^I}||\partial_{x^J}|}g_{JI}$, the super Riemann curvatures of the left
and right tangent bundles of the noncommutative super surface $X$ satisfies the first Bianchi identity and the second Bianchi identity
\begin{align}\label{a16}
(-1)^{|\partial_{x^I}||\partial_{x^K}|}R_{IJK}^L+(-1)^{|\partial_{x^I}||\partial_{x^J}|}R_{JKI}^L+(-1)^{|\partial_{x^J}||\partial_{x^K}|}R_{KIJ}^L=0,
\end{align}
\begin{align}\label{a16ooo}
(-1)^{|\partial_{x^P}||\partial_{x^J}|}R_{IJK;P}^L+(-1)^{|\partial_{x^I}||\partial_{x^J}|}R_{PIK;J}^L+(-1)^{|\partial_{x^P}||\partial_{x^I}|}R_{JPK;I}^L=0.
\end{align}
\end{thm}
\begin{proof}
Note
\begin{align}\label{a17}
Bi(\partial_{x^I},\partial_{x^J},\partial_{x^K}):=R_{IJK}^L+(-1)^{|\partial_{x^I}|(|\partial_{x^J}|+|\partial_{x^K}|)}R_{JKI}^L+(-1)^{(|\partial_{x^I}|+|\partial_{x^J}|)|\partial_{x^K}|}R_{KIJ}^L.\nonumber\\
\end{align}
By $g_{IJ}=(-1)^{|\partial_{x^I}||\partial_{x^J}|}g_{JI},$ then $T(\partial_{x^I},\partial_{x^J})=0,$ that is $\nabla_{\partial_{x^I}}\partial_{x^J}=(-1)^{|\partial_{x^I}||\partial_{x^J}|}\nabla_{\partial_{x^J}}\partial_{x^I}$. So using the definition of the Riemann curvature tensor, then
\begin{align}\label{a187}
&Bi(\partial_{x^I},\partial_{x^J},\partial_{x^K})\nonumber\\
&=[\nabla_{\partial_{x^I}},\nabla_{\partial_{x^J}}]\partial_{x^K}+(-1)^{|\partial_{x^I}|(|\partial_{x^J}|+|\partial_{x^K}|)}[\nabla_{\partial_{x^J}},\nabla_{\partial_{x^K}}]\partial_{x^I}+(-1)^{(|\partial_{x^I}|+|\partial_{x^J}|)|\partial_{x^K}|}[\nabla_{\partial_{x^K}},\nabla_{\partial_{x^I}}]\partial_{x^J}\nonumber\\
&=\nabla_{\partial_{x^I}}\nabla_{\partial_{x^J}}\partial_{x^K}-(-1)^{|\partial_{x^I}||\partial_{x^J}|}\nabla_{\partial_{x^J}}\nabla_{\partial_{x^I}}\partial_{x^K}+(-1)^{|\partial_{x^I}|(|\partial_{x^J}|+|\partial_{x^K}|)}[\nabla_{\partial_{x^J}}{\partial_{x^K}}\partial_{x^I}\nonumber\\
&-(-1)^{|\partial_{x^J}||\partial_{x^K}|}\nabla_{\partial_{x^K}}{\partial_{x^J}}\partial_{x^I}]+(-1)^{(|\partial_{x^I}|+|\partial_{x^J}|)|\partial_{x^K}|}[\nabla_{\partial_{x^K}}\nabla_{\partial_{x^I}}\partial_{x^J}-(-1)^{|\partial_{x^I}||\partial_{x^K}|}\nabla_{\partial_{x^I}}\nabla_{\partial_{x^K}}\partial_{x^J}]\nonumber\\
&=\nabla_{\partial_{x^I}}[\nabla_{\partial_{x^J}}\partial_K-(-1)^{|\partial_{x^J}||\partial_{x^K}|}\nabla_{\partial_{x^K}}\partial_{x^J}]-(-1)^{|\partial_{x^I}||\partial_{x^J}|}\nabla_{\partial_{x^J}}[\nabla_{\partial_{x^I}}\partial_{x^K}-(-1)^{|\partial_{x^I}||\partial_{x^K}|}\nabla_{\partial_{x^K}}\partial_{x^I}]\nonumber\\
&-(-1)^{(|\partial_{x^I}|+|\partial_{x^J}|)|\partial_{x^K}|}\nabla_{\partial_{x^K}}[\nabla_{\partial_{x^I}}\partial_{x^J}-(-1)^{|\partial_{x^I}||\partial_{x^J}|}\nabla_{\partial_{x^J}}\partial_{x^I}]\nonumber\\
&=0.\nonumber
\end{align}
Therefore, the first Bianchi identity holds. To prove the second Bianchi identity, let $\widetilde{\partial}^{x^L}:=\partial_{x^R}g^{RL},$ then by $[\nabla_{\partial_{x^I}},\nabla_{\partial_{x^J}}]\partial_{x^K}=R_{IJK}^L\partial_{x^L},$ we have $$R_{IJK}^L=\left<[\nabla_{\partial_{x^I}},\nabla_{\partial_{x^J}}]\partial_{x^K},\widetilde{\partial}^{x^L}\right>_g.$$
By Proposition \ref{prop2}, we have
\begin{align}
\partial_{x^P}(R_{IJK}^L)=&\partial_{x^P}\left(\left<[\nabla_{\partial_{x^I}},\nabla_{\partial_{x^J}}]\partial_{x^K},\widetilde{\partial}^{x^L}\right>_g\right)\nonumber\\
&=\left<\nabla_{\partial_{x^P}}([\nabla_{\partial_{x^I}},\nabla_{\partial_{x^J}}]\partial_{x^K}),\widetilde{\partial}^{x^L}\right>_g\nonumber\\
&+(-1)^{|\partial_{x^P}|(|\partial_{x^I}|+|\partial_{x^J}|+|\partial_{x^K}|)}\left<[\nabla_{\partial_{x^I}},\nabla_{\partial_{x^J}}]\partial_{x^K},\widetilde{\nabla}_{\partial_{x^P}}\widetilde{\partial}^{x^L}\right>_g.\nonumber\\
\end{align}
Obviously,
\begin{align}\label{a19}
&(-1)^{|\partial_{x^P}||\partial_{x^J}|}\bigg\{-\partial_{x^P}(R_{IJK}^L)+\left<\nabla_{\partial_{x^P}}([\nabla_{\partial_{x^I}},\nabla_{\partial_{x^J}}]\partial_{x^K}),\widetilde{\partial}^{x^L}\right>_g\nonumber\\
&+(-1)^{|\partial_{x^P}|(|\partial_{x^I}|+|\partial_{x^J}|+|\partial_{x^K}|)}\left<[\nabla_{\partial_{x^I}},\nabla_{\partial_{x^J}}]\partial_{x^K},\widetilde{\nabla}_{\partial_{x^P}}\widetilde{\partial}^{x^L}\right>_g\bigg\}=0.\nonumber\\
\end{align}
Replace $(P,I,J)$ with $(J,P,I),$ we get
\begin{align}\label{a20}
&(-1)^{|\partial_{x^I}||\partial_{x^J}|}\bigg\{-\partial_{x^J}(R_{PIK}^L)+\left<\nabla_{\partial_{x^J}}([\nabla_{\partial_{x^P}},\nabla_{\partial_{x^I}}]\partial_{x^K}),\widetilde{\partial}^{x^L}\right>_g\nonumber\\
&+(-1)^{|\partial_{x^J}|(|\partial_{x^P}|+|\partial_{x^I}|+|\partial_{x^K}|)}\left<[\nabla_{\partial_{x^P}},\nabla_{\partial_{x^I}}]\partial_{x^K},\widetilde{\nabla}_{\partial_{x^J}}\widetilde{\partial}^{x^L}\right>_g\bigg\}=0.\nonumber\\
\end{align}
Similarly, replace $(P,I,J)$ with $(I,J,P),$ we get
\begin{align}\label{a21}
&(-1)^{|\partial_{x^P}||\partial_{x^I}|}\bigg\{-\partial_{x^I}(R_{JPK}^L)+\left<\nabla_{\partial_{x^I}}([\nabla_{\partial_{x^J}},\nabla_{\partial_{x^P}}]\partial_{x^K}),\widetilde{\partial}^{x^L}\right>_g\nonumber\\
&+(-1)^{|\partial_{x^I}|(|\partial_{x^J}|+|\partial_{x^P}|+|\partial_{x^K}|)}\left<[\nabla_{\partial_{x^J}},\nabla_{\partial_{x^P}}]\partial_{x^K},\widetilde{\nabla}_{\partial_{x^I}}\widetilde{\partial}^{x^L}\right>_g\bigg\}=0.\nonumber\\
\end{align}
By adding the above three equations (\ref{a19}), (\ref{a20}) and (\ref{a21}) together, and by further computations. Then the second Bianchi identity can be proved.
\end{proof}
First by \cite{BG}, we get the the formula of the Ricci curvature tensor $Ric$ and scale curvature tensor $S$ in noncommutative super surface.
\begin{align}\label{b1}
Ric(\partial_{x^J},\partial_{x^K})=\sum_I(-1)^{|\partial_{x^I}|(|\partial_{x^I}|+|\partial_{x^J}|+|\partial_{x^K}|)}\frac{1}{2}[R_{IJK}^L+(-1)^{|\partial_{x^J}||\partial_{x^K}|}R_{IKJ}^L].\nonumber\\
\end{align}
\begin{align}\label{b2}
S:=\sum_{IJ}(-1)^{(|\partial_{x^I}|+1)|\partial_{x^J}|}R_{JI}\ast g^{IJ},\nonumber\\
\end{align}
where $R_{JI}$ denotes the Ricci curvature tensor $Ric(\partial_{x^J},\partial_{x^I})$.
\section{Examples}
\label{Sec:3}
In this section, we consider in some detail four concrete examples of noncommutative super surfaces.
\begin{ex}\label{ex1} Let $R^2_h(t_1,t_2)$ be double warped, and $h$ and $f$ are functions of $t_1$ and $t_2$ respectively, where $f(t_2)\neq 0,$ $h(t_1)\neq 0$, and $f_1(\overline{h})_{\overline{h}=0}=1,~f_2(\overline{h})_{\overline{h}=0}=1.$

 Then we define the metric $g$ on $R^2_h(t_1,t_2)$
\begin{align*}
g_{11}=f_1(\overline{h})f(t_2),~~~g_{12}=g_{21}=0,~~~g_{22}=f_2(\overline{h})h(t_1).
\end{align*}
Then the inverse metric is given by
\begin{align*}
g^{11}=\frac{1}{f_1(\overline{h})f(t_2)},~~~g^{12}=g^{21}=0,~~~g^{22}=\frac{1}{f_2(\overline{h})h(t_1)}.
\end{align*}
The computations are quite lengthy, thus we only record the results here. For the Christoffel symbols, by (\ref{a7000}), we have
\begin{align*}
&\Gamma_{111}=0,~~~\Gamma_{112}=-\frac{1}{2}f_1(\overline{h})\frac{\partial{f(t_2)}}{\partial{t_2}},~~~\Gamma_{121}=\frac{1}{2}f_1(\overline{h})\frac{\partial{f(t_2)}}{\partial{t_2}};\nonumber\\
&\Gamma_{122}=\frac{1}{2}f_2(\overline{h})\frac{\partial{h(t_1)}}{\partial{t_1}},~~~\Gamma_{211}=\frac{1}{2}f_1(\overline{h})\frac{\partial{f(t_2)}}{\partial{t_2}},~~~\Gamma_{212}=\frac{1}{2}f_2(\overline{h})\frac{\partial{h(t_1)}}{\partial{t_1}};\nonumber\\
&\Gamma_{221}=-\frac{1}{2}f_2(\overline{h})\frac{\partial{h(t_1)}}{\partial{t_1}},~~~\Gamma_{222}=0.\nonumber\\
\end{align*}
By (\ref{a7000}), for $R_h^2,$ the following equalities holds
\begin{align*}
&\Gamma_{11}^1=0,~~~\Gamma_{11}^2=-\frac{1}{2}\frac{f_1(\overline{h})}{f_2(\overline{h})}f'(t_2)\ast\frac{1}{h(t_1)},~~~\Gamma_{12}^1=\frac{1}{2}\frac{f'(t_2)}{f(t_2)};\nonumber\\
&\Gamma_{12}^2=\frac{1}{2}\frac{h'(t_1)}{h(t_1)},~~~\Gamma_{21}^1=\frac{1}{2}\frac{f'(t_2)}{f(t_2)},~~~\Gamma_{21}^2=\frac{1}{2}\frac{h'(t_1)}{h(t_1)};\nonumber\\
&\Gamma_{22}^1=-\frac{1}{2}\frac{f_2(\overline{h})}{f_1(\overline{h})}h'(t_1)\ast\frac{1}{f(t_2)},~~~\Gamma_{22}^2=0.
\end{align*}
We now find the super curvature
tensors with respect to $\overline{h}$,
\begin{align*}
&R_{121}^1=\frac{1}{4}\frac{h'(t_1)}{h(t_1)}\ast \frac{f'(t_2)}{f(t_2)}-\frac{1}{4}f'(t_2)\ast\frac{h'(t_1)}{h(t_1)}\ast\frac{1}{f(t_2)};\\
&R_{122}^1=\frac{-\frac{1}{2}f''(t_2)f(t_2)+\frac{1}{4}f'(t_2)^2}{f(t_2)^2}+\frac{1}{2}\frac{f_2(\overline{h})}{f_1(\overline{h})}\left[\left(\frac{1}{2}\frac{h'(t_1)^2}{h(t_1)}-h''(t_1)\right)\ast\frac{1}{f(t_2)}\right];\\
&R_{121}^2=\frac{1}{2}\frac{h''(t_1)h(t_1)-\frac{1}{2}h'(t_1)^2}{h(t_1)^2}+\frac{1}{2}\frac{f_1(\overline{h})}{f_2(\overline{h})}\left[\frac{2f''(t_2)f(t_2)-f'(t_2)^2}{2f(t_2)}\ast\frac{1}{h(t_1)}\right];\\
&R_{122}^2=\frac{1}{4}h'(t_1)\ast \frac{f'(t_2)}{f(t_2)}\ast\frac{1}{h(t_1)}-\frac{1}{4}\frac{f'(t_2)}{f(t_2)}\ast\frac{h'(t_1)}{h(t_1)}.\\
\end{align*}
By (\ref{b1}), we can also compute the Ricci curvature tensor,
\begin{align*}
&R_{11}=-\frac{1}{2}\frac{h''(t_1)h(t_1)-\frac{1}{2}h'(t_1)^2}{h(t_1)^2}+\frac{1}{2}\frac{f_1(\overline{h})}{f_2(\overline{h})}\left[\frac{2f''(t_2)f(t_2-f'(t_2)^2)}{2f(t_2)}\ast\frac{1}{h(t_1)}\right];\\
&R_{12}=R_{21}=\frac{1}{2}\bigg[\frac{1}{4}\frac{h'(t_1)}{h(t_1)}\ast \frac{f'(t_2)}{f(t_2)}-\frac{1}{4}f'(t_2)\ast\frac{h'(t_1)}{h(t_1)}\ast\frac{1}{f(t_2)}-\frac{1}{4}h'(t_1)\ast \frac{f'(t_2)}{f(t_2)}\ast\frac{1}{h(t_1)};\\
&+\frac{1}{4}\frac{f'(t_2)}{f(t_2)}\ast\frac{h'(t_1)}{h(t_1)}\bigg];\\
&R_{22}=\frac{-\frac{1}{2}f''(t_2)f(t_2)+\frac{1}{4}f'(t_2)^2}{f(t_2)^2}+\frac{1}{2}\frac{f_2(\overline{h})}{f_1(\overline{h})}\left[\left(\frac{1}{2}\frac{h'(t_1)^2}{h(t_1)}-h''(t_1)\right)\ast\frac{1}{f(t_2)}\right].\\
\end{align*}
By (\ref{b2}), we get the scalar curvature
\begin{align*} S&=-\bigg\{\frac{1}{2}\frac{h''(t_1)h(t_1)-\frac{1}{2}h'(t_1)^2}{h(t_1)^2}+\frac{1}{2}\frac{f_1(\overline{h})}{f_2(\overline{h})}\bigg[\frac{2f''(t_2)f(t_2-f'(t_2)^2)}{2f(t_2)}\ast\frac{1}{h(t_1)}\bigg]\bigg\}\ast\frac{1}{f_1(\overline{h})f(t_2)}\\
&+\bigg\{\frac{1}{4}h'(t_1)\ast \frac{f'(t_2)}{f(t_2)}\ast\frac{1}{h(t_1)}-\frac{1}{4}\frac{f'(t_2)}{f(t_2)}\ast\frac{h'(t_1)}{h(t_1)}\bigg\}\ast\frac{1}{f_2(\overline{h})h(t_1)}.\\
\end{align*}
\end{ex}
\begin{ex}\label{ex2}
 Let $f(t_2)>0,~h(t_1)>0,$ and the metric $g$ satisfies the following equalities
\begin{align*}
g_{11}=f(t_2),~~~g_{12}=0,~~~g_{21}=1,~~~g_{22}=h(t_1).
\end{align*}
Then the inverse metric is given by
\begin{align*}
g^{11}=\frac{1}{f(t_2)},~~~g^{12}=0,~~~g^{21}=-\frac{1}{h(t_1)}\ast\frac{1}{f(t_2)},~~~g^{22}=\frac{1}{h(t_1)}.
\end{align*}
Then the computations can
be carried out in much the same way as in the case of Example \ref{ex1}, and by (\ref{a7000}) we mainly
get
\begin{align*}
&\Gamma_{111}=0,~~~\Gamma_{112}=-\frac{1}{2}\frac{\partial{f(t_2)}}{\partial{t_2}},~~~\Gamma_{121}=\frac{1}{2}\frac{\partial{f(t_2)}}{\partial{t_2}};\nonumber\\
&\Gamma_{122}=\frac{1}{2}\frac{\partial{h(t_1)}}{\partial{t_1}},~~~\Gamma_{211}=\frac{1}{2}\frac{\partial{f(t_2)}}{\partial{t_2}},~~~\Gamma_{212}=\frac{1}{2}\frac{\partial{h(t_1)}}{\partial{t_1}};\nonumber\\
&\Gamma_{221}=-\frac{1}{2}\frac{\partial{h(t_1)}}{\partial{t_1}},~~~\Gamma_{222}=0.
\end{align*}
By (\ref{a7000}), the following equalities holds
\begin{align*}
&\Gamma_{11}^1=\frac{1}{2}f'(t_2)\ast\frac{1}{h(t_1)}\ast\frac{1}{f(t_2)},~~~\Gamma_{11}^2=-\frac{1}{2}f'(t_2)\ast\frac{1}{h(t_1)},~~~\Gamma_{12}^2=\frac{1}{2}\frac{h'(t_1)}{h(t_1)};\nonumber\\
&\Gamma_{12}^1=-\frac{1}{2}\ast\frac{h'(t_1)}{h(t_1)}\ast\frac{1}{f(t_2)}+\frac{1}{2}\frac{f'(t_2)}{f(t_2)},~~~\Gamma_{21}^1=\frac{1}{2}\frac{f'(t_2)}{f(t_2)}-\frac{1}{2}\frac{h'(t_1)}{h(t_1)}\ast\frac{1}{f(t_2)};\nonumber\\
&\Gamma_{21}^2=\frac{1}{2}\frac{h'(t_1)}{h(t_1)},~~~\Gamma_{22}^1=-\frac{1}{2}h'(t_1)\ast\frac{1}{f(t_2)},~~~\Gamma_{22}^2=0.
\end{align*}
We now find the super curvature
tensors with respect to $\overline{h}$,
\begin{align*}
&R_{121}^1=\frac{h'(t_1)^2-2h(t_1)h''(t_1)}{4h(t_1)^2}\ast \frac{1}{f(t_2)}-\frac{f'(t_2)^2-2f''(t_2)f(t_2)}{4f(t_2)^2}\ast\frac{1}{h(t_1)}\ast\frac{1}{f(t_2)}\\
&+\frac{1}{4}f'(t_2)\ast\frac{1}{h(t_1)}\ast\frac{f'(t_2)}{f(t_2)^2}-\frac{1}{4}\frac{h'(t_1)}{h(t_1)}\ast\frac{f'(t_2)}{f(t_2)}\ast\frac{1}{h(t_1)}\ast\frac{1}{f(t_2)}+\frac{1}{4}\frac{h'(t_1)}{h(t_1)}\ast\frac{f'(t_2)}{f(t_2)}\\
&+\frac{1}{4}f'(t_2)\ast\frac{1}{h(t_1)}\ast\frac{1}{f(t_2)}\ast\frac{h'(t_1)}{h(t_1)}\ast\frac{1}{f(t_2)}-\frac{1}{4}f'(t_2)\ast\frac{h'(t_1)}{h(t_1)}\ast\frac{1}{f(t_2)};\\
&R_{122}^1=\frac{h'(t_1)^2-2h(t_1)h''(t_1)}{4h(t_1)}\ast \frac{1}{f(t_2)}-\frac{1}{4}\frac{h'(t_1)}{h(t_1)}\ast\frac{f'(t_2)}{f(t_2)^2}+\frac{f'(t_2)^2-2f(t_2)f''(t_2)}{4f(t_2)^2}\\
&-\frac{1}{4}h'(t_1)\ast\frac{f'(t_2)}{f(t_2)}\ast\frac{1}{h(t_1)}\ast\frac{1}{f(t_2)}-\frac{1}{4}\frac{h'(t_1)}{h(t_1)}\ast\frac{1}{f(t_2)}\ast\frac{h'(t_1)}{h(t_1)}\ast\frac{1}{f(t_2)}\\
&+\frac{1}{4}\ast\frac{f'(t_2)}{f(t_2)}\ast\frac{h'(t_1)}{h(t_1)}\ast\frac{1}{f(t_2)};\\
&R_{121}^2=\frac{2h(t_1)h''(t_1)-h'(t_1)^2}{4h(t_1)^2}+\frac{2f(t_2)f''(t_2)-f'(t_2)^2}{4f(t_2)}\ast\frac{1}{h(t_1)}+\frac{1}{4}\frac{h'(t_1)}{h(t_1)}\\
&\ast\frac{f'(t_2)}{f(t_2)}\ast\frac{1}{h(t_1)}-\frac{1}{4}f'(t_2)\ast\frac{1}{h(t_1)}\ast\frac{1}{f(t_2)}\ast\frac{h'(t_1)}{h(t_1)};\\
&R_{122}^2=\frac{1}{4}h'(t_1)\ast \frac{f'(t_2)}{f(t_2)}\ast\frac{1}{h(t_1)}-\frac{1}{4}\frac{h'(t_1)}{h(t_1)}\ast\frac{f'(t_2)}{f(t_2)}\ast\frac{1}{h(t_1)}+\frac{f'(t_2)^2}{f(t_2)}\ast\frac{1}{h(t_1)}\\
&-\frac{1}{4}\frac{h'(t_1)^2}{h(t_1)^2}.
\end{align*}
By (\ref{b1}), we can also compute the Ricci curvature tensor,
\begin{align*}
&R_{11}=-\frac{2h(t_1)h''(t_1)-h'(t_1)^2}{4h(t_1)^2}-\frac{2f(t_2)f''(t_2)-f'(t_2)^2}{4f(t_2)}\ast\frac{1}{h(t_1)}-\frac{1}{4}\frac{h'(t_1)}{h(t_1)}\\
&\ast\frac{f'(t_2)}{f(t_2)}\ast\frac{1}{h(t_1)}+\frac{1}{4}f'(t_2)\ast\frac{1}{h(t_1)}\ast\frac{1}{f(t_2)}\ast\frac{h'(t_1)}{h(t_1)};\\
&R_{12}=R_{21}=\frac{1}{2}\bigg[\frac{h'(t_1)^2-2h(t_1)h''(t_1)}{4h(t_1)^2}\ast \frac{1}{f(t_2)}-\frac{f'(t_2)^2-2f''(t_2)f(t_2)}{4f(t_2)^2}\ast\frac{1}{h(t_1)}\ast\frac{1}{f(t_2)}\\
&-\frac{1}{4}h'(t_1)\ast \frac{f'(t_2)}{f(t_2)}\ast\frac{1}{h(t_1)}+\frac{1}{4}\frac{h'(t_1)}{h(t_1)}\ast\frac{f'(t_2)}{f(t_2)}\ast\frac{1}{h(t_1)}-\frac{f'(t_2)^2}{f(t_2)}\ast\frac{1}{h(t_1)}+\frac{1}{4}\frac{h'(t_1)^2}{h(t_1)^2}\bigg];\\
&R_{22}=\frac{h'(t_1)^2-2h(t_1)h''(t_1)}{4h(t_1)}\ast \frac{1}{f(t_2)}-\frac{1}{4}\frac{h'(t_1)}{h(t_1)}\ast\frac{f'(t_2)}{f(t_2)^2}+\frac{f'(t_2)^2-2f(t_2)f''(t_2)}{4f(t_2)^2}\\
&-\frac{1}{4}h'(t_1)\ast\frac{f'(t_2)}{f(t_2)}\ast\frac{1}{h(t_1)}\ast\frac{1}{f(t_2)}-\frac{1}{4}\frac{h'(t_1)}{h(t_1)}\ast\frac{1}{f(t_2)}\ast\frac{h'(t_1)}{h(t_1)}\ast\frac{1}{f(t_2)}\\
&+\frac{1}{4}\ast\frac{f'(t_2)}{f(t_2)}\ast\frac{h'(t_1)}{h(t_1)}\ast\frac{1}{f(t_2)}.
\end{align*}
By (\ref{b2}), we get the scalar curvature
\begin{align*} S&=\bigg\{-\frac{2h(t_1)h''(t_1)-h'(t_1)^2}{4h(t_1)^2}-\frac{2f(t_2)f''(t_2)-f'(t_2)^2}{4f(t_2)}\ast\frac{1}{h(t_1)}-\frac{1}{4}\frac{h'(t_1)}{h(t_1)}\ast\frac{f'(t_2)}{f(t_2)}\ast\frac{1}{h(t_1)}\\
&+\frac{1}{4}f'(t_2)\ast\frac{1}{h(t_1)}\ast\frac{1}{f(t_2)}\ast\frac{h'(t_1)}{h(t_1)}\bigg\}\ast\frac{1}{f(t_2)}+\bigg\{\frac{1}{2}\bigg[\frac{h'(t_1)^2-2h(t_1)h''(t_1)}{4h(t_1)^2}\ast \frac{1}{f(t_2)}\\
&-\frac{f'(t_2)^2-2f''(t_2)f(t_2)}{4f(t_2)^2}\ast\frac{1}{h(t_1)}\ast\frac{1}{f(t_2)}-\frac{1}{4}h'(t_1)\ast \frac{f'(t_2)}{f(t_2)}\ast\frac{1}{h(t_1)}+\frac{1}{4}\frac{h'(t_1)}{h(t_1)}\ast\frac{f'(t_2)}{f(t_2)}\ast\frac{1}{h(t_1)}\\
&-\frac{f'(t_2)^2}{f(t_2)}\ast\frac{1}{h(t_1)}+\frac{1}{4}\frac{h'(t_1)^2}{h(t_1)^2}\bigg]\bigg\}\ast\bigg(-\frac{1}{h(t_1)}\ast\frac{1}{f(t_2)}\bigg)+\bigg\{\frac{h'(t_1)^2-2h(t_1)h''(t_1)}{4h(t_1)}\ast \frac{1}{f(t_2)}\\
&-\frac{1}{4}\frac{h'(t_1)}{h(t_1)}\ast\frac{f'(t_2)}{f(t_2)^2}+\frac{f'(t_2)^2-2f(t_2)f''(t_2)}{4f(t_2)^2}-\frac{1}{4}h'(t_1)\ast\frac{f'(t_2)}{f(t_2)}\ast\frac{1}{h(t_1)}\ast\frac{1}{f(t_2)}\\
&-\frac{1}{4}\frac{h'(t_1)}{h(t_1)}\ast\frac{1}{f(t_2)}\ast\frac{h'(t_1)}{h(t_1)}\ast\frac{1}{f(t_2)}+\frac{1}{4}\ast\frac{f'(t_2)}{f(t_2)}\ast\frac{h'(t_1)}{h(t_1)}\ast\frac{1}{f(t_2)}\bigg\}\ast\frac{1}{h(t_1)}.\\
\end{align*}
\end{ex}
\begin{ex}\label{ex3}Another simple example is described by $A=(\frac{\partial}{\partial t_1},\frac{\partial}{\partial t_2},\frac{\partial}{\partial \xi}),$ and let the metric $g$ is even and asymmetry, then $g(\frac{\partial}{\partial \xi},\frac{\partial}{\partial \xi})\neq0.$
And the metric $g$ satisfies
\begin{align*}
&g_{11}=\phi_1(\overline{h})f_1(t_1),~~~g_{12}=0,~~~g_{13}=\varepsilon\xi,~~~g_{21}=0;\\
&g_{22}=\phi_2(\overline{h})f_2(t_1),~~~g_{23}=0,~~~g_{31}=\varepsilon\xi,~~~g_{32}=0,~~~g_{33}=1,
\end{align*}
where $f_1(t_1)>0,~f_2(t_1)>0,$~~~$\phi_1(\overline{h})_{\overline{h}=0}=1,$ and $\phi_2(\overline{h})_{\overline{h}=0}=1.$

Then the inverse metric is given by
\begin{align*}
&g^{11}=\frac{1}{\phi_1(\overline{h})f_1(t_1)},~~~g^{12}=0,~~~g^{13}=-\varepsilon\xi\frac{1}{\phi_1(\overline{h})f_1(t_1)},~~~g^{21}=0;\\
&g^{22}=\frac{1}{\phi_2(\overline{h})f_2(t_1)},~~~g^{23}=0,~~~g^{31}=-\varepsilon\xi\frac{1}{\phi_1(\overline{h})f_1(t_1)},~~~g^{32}=0,~~~g^{33}=1.
\end{align*}
Then by (\ref{a7000}), we have
\begin{align*}
&\Gamma_{111}=\frac{1}{2}\phi_1(\overline{h})f_1'(t_1),~~~\Gamma_{122}=\frac{1}{2}\phi_2(\overline{h})f_2'(t_1),~~~\Gamma_{133}=\varepsilon;\nonumber\\
&\Gamma_{212}=\frac{1}{2}\phi_2(\overline{h})f_2'(t_1),~~~\Gamma_{221}=-\frac{1}{2}\phi_2(\overline{h})f_2'(t_1),~~~\Gamma_{313}=\varepsilon,~~~\Gamma_{331}=0,
\end{align*}
where $\Gamma_{IJK}=0$ for else $(I,J,K).$

By (\ref{a7000}), the following equalities holds
\begin{align*}
&\Gamma_{11}^1=\frac{1}{2}\frac{f_1'(t_1)}{f_1(t_1)},~~~\Gamma_{11}^2=0,~~~\Gamma_{11}^3=-\frac{1}{2}\varepsilon\xi\frac{f_1'(t_1)}{f_1(t_1)},~~~\Gamma_{12}^1=0,~~~\Gamma_{12}^2=\frac{1}{2}\frac{f_2'(t_1)}{f_2(t_1)};\nonumber\\
&\Gamma_{12}^3=0,~~~\Gamma_{13}^1=-\varepsilon^2\xi\frac{1}{\phi_1(\overline{h})}\frac{1}{f_1(t_1)},~~~\Gamma_{13}^2=0,~~~\Gamma_{13}^3=\varepsilon;\nonumber\\
&\Gamma_{21}^1=0,~~~\Gamma_{21}^2=\frac{1}{2}\frac{f_2'(t_1)}{f_2(t_1)},~~~\Gamma_{21}^3=0,~~~\Gamma_{22}^1=-\frac{1}{2}\frac{\phi_2(\overline{h})}{\phi_1(\overline{h})}\frac{f_2'(t_1)}{f_1(t_1)},~~~\Gamma_{22}^2=0;\nonumber\\
&\Gamma_{22}^3=\frac{1}{2}\varepsilon\xi\frac{\phi_2(\overline{h})}{\phi_1(\overline{h})}\frac{f_2'(t_1)}{f_1(t_1)},~~~\Gamma_{23}^1=\Gamma_{23}^2=\Gamma_{23}^3=0,~~~\Gamma_{31}^1=-\varepsilon^2\xi\frac{1}{\phi_1(\overline{h})}\frac{1}{f_1(t_1)},~~~\Gamma_{31}^2=0;\nonumber\\
&\Gamma_{31}^3=\varepsilon,~~~\Gamma_{32}^1=\Gamma_{32}^2=\Gamma_{32}^3=\Gamma_{33}^1=\Gamma_{33}^2=\Gamma_{33}^3=0.
\end{align*}
We now find nonzero components of the super curvature
tensors with respect to $\overline{h}$,
\begin{align*}
&R_{122}^1=-\frac{1}{4}\frac{\phi_2(\overline{h})}{\phi_1(\overline{h})}\frac{2f_2''(t_1)f_1(t_1)-f_2'(t_1)f_1'(t_1)}{f_1(t_1)^2}+\frac{1}{4}\frac{\phi_2(\overline{h})}{\phi_1(\overline{h})}\frac{f_2'(t_1)^2}{f_1(t_1)f_2(t_1)};\\
&R_{121}^2=\frac{2f_2''(t_1)f_2(t_1)-f_2'(t_1)^2}{4f_2(t_1)^2}-\frac{1}{4}\frac{f_1'(t_1)f_2'(t_1)}{f_1(t_1)f_2(t_1)},~~~R_{123}^2=-\frac{1}{2}\varepsilon^2\xi\frac{1)}{\phi_1(\overline{h})}\frac{f_2'(t_1)}{f_1(t_1)f_2(t_1)};\\
&R_{122}^3=\frac{1}{2}\varepsilon\xi\frac{\phi_2(\overline{h})}{\phi_1(\overline{h})}\frac{f_2''(t_1)}{f_1(t_1)}+\frac{1}{4}\varepsilon\xi\frac{\phi_2(\overline{h})}{\phi_1(\overline{h})}\frac{2\varepsilon f_2'(t_1)-f_2'(t_1)^2}{f_1(t_1)f_2(t_1)},~~~R_{131}^1=-2\varepsilon^2\xi\frac{1}{\phi_1(\overline{h})}\frac{f_1'(t_1)}{f_1(t_1)^2}\nonumber\\
&-\varepsilon^3\xi\frac{1}{\phi_1(\overline{h})}\frac{1}{f_1(t_1)},~~~R_{133}^1=\varepsilon^2\frac{1}{\phi_1(\overline{h})}\frac{1}{f_1(t_1)},~~~R_{133}^3=-\varepsilon^3\xi\frac{1}{\phi_1(\overline{h})}\frac{1}{f_1(t_1)},~~~R_{131}^3=\varepsilon^2;\nonumber\\
&R_{231}^2=-\frac{1}{2}\varepsilon^2\xi\frac{\phi_2(\overline{h})}{\phi_1(\overline{h})^2}\frac{f_2'(t_1)}{f_1(t_1)^2},~~~R_{232}^2=-\frac{1}{2}\varepsilon^2\xi\frac{1}{\phi_1(\overline{h})^2}\frac{f_2'(t_1)}{f_1(t_1)f_2(t_1)},~~~R_{331}^1=-2\varepsilon^2\frac{1)}{\phi_1(\overline{h})}\frac{1}{f_1(t_1)},
\end{align*}
where $R_{IJK}^L=0,$ for any else $(I,J,K,L)$.\\
By (\ref{b1}), we can also compute asymptotic expansions of the Ricci curvature tensor,
\begin{align*}
&R_{11}=\frac{f_2'(t_1)^2}{4f_2(t_1)^2-2f_2''(t_1)f_2(t_1)}+\frac{1}{4}\frac{f_1'(t_1)f_2'(t_1)}{f_1(t_1)f_2(t_1)}+\varepsilon^2,~~~R_{12}=0;\\
&R_{13}=\varepsilon^2\xi\frac{1}{\phi_1(\overline{h})}\bigg(-\frac{f_1'(t_1)}{f_1(t_1)^2}+\frac{1}{f_1(t_1)}\bigg),~~~R_{23}=R_{33=0};\\
&R_{22}=-\frac{1}{4}\frac{\phi_2(\overline{h})}{\phi_1(\overline{h})}\frac{2f_2''(t_1)f_1(t_1)-f_2'(t_1)f_1'(t_1)}{f_1(t_1)^2}+\frac{1}{4}\frac{\phi_2(\overline{h})}{\phi_1(\overline{h})}\frac{f_2'(t_1)^2}{f_1(t_1)f_2(t_1)}.
\end{align*}
By (\ref{b2}), we get the scalar curvature
\begin{align*} S&=\bigg[\frac{f_2'(t_1)^2}{4f_2(t_1)^2-2f_2''(t_1)f_2(t_1)}+\frac{1}{4}\frac{f_1'(t_1)f_2'(t_1)}{f_1(t_1)f_2(t_1)}+\varepsilon^2\bigg]\ast\frac{1}{\phi_1(\overline{h})}\frac{1}{f_1(t_1)}\\
&+\bigg[-\frac{1}{4}\frac{1}{\phi_1(\overline{h})}\frac{2f_2''(t_1)f_1(t_1)-f_2'(t_1)f_1'(t_1)}{f_1(t_1)^2}+\frac{1}{4}\frac{1}{\phi_1(\overline{h})}\frac{f_2'(t_1)^2}{f_1(t_1)f_2(t_1)}\bigg]\ast\frac{1}{f_2(t_1)}.\\
\end{align*}
\end{ex}
\begin{ex}\label{ex4}The last example is described by $A=(\frac{\partial}{\partial t_1},\frac{\partial}{\partial t_2},\frac{\partial}{\partial \xi}),$ and let the metric $g$ is even and asymmetry, then $g(\frac{\partial}{\partial \xi},\frac{\partial}{\partial \xi})\neq0.$
And the metric $g$ satisfies
\begin{align*}
&g_{11}=f(t_2),~~~g_{12}=0,~~~g_{13}=\varepsilon\xi,~~~g_{21}=0;\\
&g_{22}=h(t_1),~~~g_{23}=0,~~~g_{31}=0,~~~g_{32}=0,~~~g_{33}=1.
\end{align*}
Then the inverse metric is given by
\begin{align*}
&g^{11}=\frac{1}{f(t_2)},~~~g^{12}=0,~~~g^{13}=-\varepsilon\xi\frac{1}{f(t_2)},~~~g^{21}=0;\\
&g^{22}=\frac{1}{h(t_1)},~~~g^{23}=0,~~~g^{31}=0,~~~g^{32}=0,~~~g^{33}=1.
\end{align*}
Then by (\ref{a7000}), we have
\begin{align*}
&\Gamma_{112}=-\frac{1}{2}f'(t_2),~~~\Gamma_{121}=\frac{1}{2}f'(t_2),~~~\Gamma_{122}=\frac{1}{2}h'(t_1),~~~\Gamma_{133}=\varepsilon;\nonumber\\
&\Gamma_{211}=\frac{1}{2}f'(t_2),~~~\Gamma_{212}=\frac{1}{2}h'(t_1),~~~\Gamma_{221}=-\frac{1}{2}h'(t_1),~~~\Gamma_{313}=\frac{1}{2}\varepsilon,
\end{align*}
where $\Gamma_{IJK}=0$ for else $(I,J,K).$

By (\ref{a7000}), the following equalities holds
\begin{align*}
&\Gamma_{11}^1=0,~~~\Gamma_{11}^2=-\frac{1}{2}f'(t_2)\ast\frac{1}{h(t_1)},~~~\Gamma_{11}^3=0,~~~\Gamma_{12}^1=\frac{1}{2}\frac{f'(t_2)}{f(t_2)},~~~\Gamma_{12}^2=\frac{1}{2}\frac{h'(t_1)}{h(t_1)};\nonumber\\
&\Gamma_{12}^3=-\frac{1}{2}\varepsilon\xi\frac{f'(t_2)}{f(t_2)},~~~\Gamma_{13}^1=\Gamma_{13}^2=0,~~~\Gamma_{13}^3=\varepsilon,~~~\Gamma_{21}^1=\frac{1}{2}\frac{f'(t_2)}{f(t_2)},~~~\Gamma_{21}^2=\frac{1}{2}\frac{h'(t_1)}{h(t_1)};\nonumber\\
&\Gamma_{21}^3=-\frac{1}{2}\varepsilon\xi\frac{f'(t_2)}{f(t_2)},~~~\Gamma_{22}^1=-\frac{1}{2}h'(t_1)\ast\frac{1}{f(t_2)},~~~\Gamma_{22}^2=0,~~~\Gamma_{22}^3=\frac{1}{2}\varepsilon\xi h'(t_1)\ast\frac{1}{f(t_2)};\nonumber\\
&\Gamma_{23}^1=\Gamma_{23}^2=\Gamma_{23}^3=\Gamma_{31}^3=\Gamma_{32}^1=\Gamma_{32}^2=\Gamma_{32}^3=\Gamma_{33}^1=\Gamma_{33}^2=\Gamma_{33}^3=0.\nonumber\\
\end{align*}
We now find the super curvature
tensors with respect to $\overline{h}$,
\begin{align*}
&R_{121}^1=\frac{1}{4}\frac{h'(t_1)}{h(t_1)}\ast\frac{f'(t_2)}{f(t_2)}-\frac{1}{4}f'(t_2)\ast\frac{h'(t_1)}{h(t_1)}\ast\frac{1}{f(t_2)},~~~R_{122}^1=\frac{h'(t_1)^2-2h''(t_1)h(t_1)}{4h(t_1)}\ast\frac{1}{f(t_2)}\nonumber\\
&+\frac{f'(t_2)^2-2f''(t_2)f(t_2)}{4f(t_2)^2},~~~R_{121}^2=\frac{h''(t_1)h(t_1)-h'(t_1)^2}{4h(t_1)^2}+\frac{2f''(t_2)f(t_2)f'(t_2)^2}{4f(t_2)}\ast\frac{1}{h(t_1)};\nonumber\\
&R_{122}^2=\frac{1}{2}h'(t_1)\ast\frac{f'(t_2)}{f(t_2)}\ast\frac{1}{h(t_1)}-\frac{1}{4}\frac{f'(t_2)}{f(t_2)}\ast\frac{h'(t_1)}{h(t_1)},~~~R_{121}^3=-\frac{1}{4}\varepsilon\xi\frac{h'(t_1)}{h(t_1)}\ast\frac{f'(t_2)}{f(t_2)}-\frac{1}{2}\xi\frac{f'(t_2)}{f(t_2)}\nonumber\\
&+\frac{1}{4}\varepsilon\xi f'(t_2)\ast\frac{h'(t_1)}{h(t_1)}\ast\frac{1}{f(t_2)},~~~R_{122}^3=\frac{1}{2}\xi\bigg(\varepsilon h''(t_1)+h'(t_1)-\frac{1}{2}\varepsilon\frac{h'(t_1)^2}{h(t_1)}\bigg)\ast\frac{1}{f(t_2)}\nonumber\\
&+\frac{1}{4}\varepsilon\xi\frac{2f''(t_2)f(t_2)-f'(t_2)^2}{f(t_2)^2},~~~R_{131}^3=\frac{1}{2}\varepsilon^2,~~~R_{132}^3=\frac{1}{4}\varepsilon\frac{f'(t_2)}{f(t_2)},~~~R_{231}^3=\frac{1}{4}\varepsilon\frac{f'(t_2)}{f(t_2)};\nonumber\\
&R_{232}^3=-\frac{1}{4}\varepsilon h'(t_1)\ast\frac{1}{f(t_2)},
\end{align*}
where $R_{IJK}^L=0,$ for any else $(I,J,K,L)$.\\
By (\ref{b1}), we can also compute the Ricci curvature tensor,
\begin{align*}
&R_{11}=-\frac{h''(t_1)h(t_1)-h'(t_1)^2}{4h(t_1)^2}-\frac{2f''(t_2)f(t_2)f'(t_2)^2}{4f(t_2)}\ast\frac{1}{h(t_1)}+\frac{1}{2}\varepsilon^2,~~~R_{12}=\frac{1}{8}\frac{h'(t_1)}{h(t_1)}\ast\frac{f'(t_2)}{f(t_2)}\\
&-\frac{1}{8}f'(t_2)\ast\frac{h'(t_1)}{h(t_1)}\ast\frac{1}{f(t_2)}\frac{1}{4}h'(t_1)\ast\frac{f'(t_2)}{f(t_2)}\ast\frac{1}{h(t_1)}+\frac{1}{8}\frac{f'(t_2)}{f(t_2)}\ast\frac{h'(t_1)}{h(t_1)}+\frac{1}{4}\varepsilon\frac{f'(t_2)}{f(t_2)},~~~R_{13}=0;\\
&R_{22}=\frac{h'(t_1)^2-2h''(t_1)h(t_1)}{4h(t_1)}\ast\frac{1}{f(t_2)}+\frac{f'(t_2)^2-2f''(t_2)f(t_2)}{4f(t_2)^2}-\frac{1}{4}\varepsilon h'(t_1)\ast\frac{1}{f(t_2)};\\
&R_{23}=R_{33=0}.
\end{align*}
By (\ref{b2}), we get the scalar curvature
\begin{align*} S&=\bigg[-\frac{h''(t_1)h(t_1)-h'(t_1)^2}{4h(t_1)^2}-\frac{2f''(t_2)f(t_2)f'(t_2)^2}{4f(t_2)}\ast\frac{1}{h(t_1)}+\frac{1}{2}\varepsilon^2\bigg]\ast\frac{1}{\phi_1(\overline{h})}\frac{1}{f(t_2)}\\
&+\bigg[\frac{h'(t_1)^2-2h''(t_1)h(t_1)}{4h(t_1)}\ast\frac{1}{f(t_2)}+\frac{f'(t_2)^2-2f''(t_2)f(t_2)}{4f(t_2)^2}-\frac{1}{4}\varepsilon h'(t_1)\ast\frac{1}{f(t_2)}\bigg]\ast\frac{1}{h(t_1)}.\\
\end{align*}
\end{ex}

\section{Acknowledgements}

The first author was supported in part by  NSFC No.11771070.

\vskip 0.5 true cm


\bigskip

\noindent {\footnotesize {\it Yong Wang} \\
{School of Mathematics and Statistics, Northeast Normal University, Changchun 130024, China}\\
{Email: wangy581@nenu.edu.cn}

\noindent {\footnotesize {\it T. Wu} \\
{School of Mathematics and Statistics, Northeast Normal University, Changchun 130024, China}\\
{Email: wut977@nenu.edu.cn}

\end{document}